\documentclass{amsart}
\usepackage{amsmath,amsthm,amscd,amssymb}

\usepackage[colorlinks=true, urlcolor=NavyBlue, linkcolor=NavyBlue, citecolor=NavyBlue]{hyperref}

\usepackage{graphicx}
\usepackage[percent]{overpic}

\usepackage[dvipsnames,usenames]{color}
\usepackage{tikz}

\usepackage{subfigure}

\newtheorem{thm}{Theorem}[section]
\newtheorem{lem}[thm]{Lemma}
\newtheorem{prop}[thm]{Proposition}

\theoremstyle{definition}

\numberwithin{equation}{section}
\numberwithin{figure}{section}

\newcommand{\inv}{^{-1}}

\newcommand{\del}{\partial}
\renewcommand{\a}{\alpha}

\newcommand{\calc}{{\mathcal{C}}}

\newcommand{\bbz}{{\mathbb{Z}}}
\newcommand{\bbq}{{\mathbb{Q}}}

\newcommand{\calt}{{\mathcal{T}}}
\newcommand{\caln}{{\mathcal{N}}}
\newcommand{\calp}{{\mathcal{P}}}

\newcommand{\cfkinfty}{CFK^\infty}

\newcommand{\sss}{{\mathfrak{s}}}
\newcommand{\spinc}{\mathrm{spin^c}}
\newcommand{\spin}{\mathrm{spin}}
\newcommand{\consum}{\mathop{\Large {\#}}}

\title{Structure in the Bipolar Filtration of Topologically Slice Knots}

\author{Tim D. Cochran$^\dag$}
\address{Department of Mathematics MS-136, P.O. Box 1892, Rice University, Houston, TX 77251-1892}
\email{cochran@rice.edu}
\urladdr{http://math.rice.edu/~cochran}
\thanks{$\dag$ Partially supported by the National Science Foundation DMS-1006908}

\author{Peter D. Horn$^\ddag$}
\address{Department of Mathematics\\Syracuse University\\215 Carnegie Building\\Syracuse, NY 13244-1150}
\email{pdhorn@syr.edu}
\urladdr{http://pdhorn.mysite.syr.edu/}
\thanks{$\ddag$ Partially supported by the NSF DMS-1205922 and NSF Postdoctoral Fellowship DMS-0902786}

\begin{document}

\begin{abstract}
	
	Let $\mathcal{T}$ be the group of smooth concordance classes of topologically slice knots and suppose 
 \[ \cdots \subset \calt_{n+1} \subset \calt_n \subset \cdots \calt_2 \subset \calt_1 \subset \calt_0 \subset \calt \] 
is the  \textit{bipolar filtration of}  $\calt$. We show that $\mathcal{T}_0/\mathcal{T}_1$ has infinite rank, even modulo Alexander polynomial one knots.  Recall that knots in $\mathcal{T}_0$ (a topologically slice $0$-bipolar knot) necessarily have zero $\tau$, $s$ and $\epsilon$-invariants. Our invariants are detected using certain $d$-invariants associated to the $2$-fold branched covers.
\end{abstract}


\maketitle

\section{Introduction}\label{sec:Introduction}

Research into $4$-dimensional manifolds in the last $30$ years  has revealed a vast difference between the topological and smooth categories. There is a purely local paradigm for this disparity in terms of knot theory. Given a \textit{knot} $K$ in $S^3\equiv \partial B^4$, it can happen that there is a $2$-disk, $\Delta$, topologically  embedded in the $4$-ball so that $\partial \Delta=K$ but there exists no smoothly embedded disk with $K$ as boundary.  This motivates the study of which knots in $S^3$ bound embedded $2$-disks in $B^4$ in the topological category, but not the smooth category. This paper contributes to the on-going efforts to quantify and classify this difference in categories.

To be more precise, the set of oriented knots in $S^3$ comes equipped with a binary operation called connected sum (denoted $\#$).  Given a knot $K$ let $-K$ denote the knot obtained by taking the mirror-image and reversing the circle's orientation.  Two knots $K$ and $J$ are \textit{smoothly concordant} if there is an annulus $A$ properly and smoothly embedded in $S^3\times[0,1]$ with $A \cap S^3\times\{1\} = K$ and $A  \cap S^3\times\{0\} = -J$.  The set of equivalence classes is a group, called the \emph{smooth knot concordance group}, denoted $\calc$.  The inverse of $[K]$ is $[-K]$, and the identity element is the class of \emph{slice knots}, i.e. knots which bound properly and smoothly embedded disks in the four-ball. There is a weaker equivalence relation that determines the \emph{topological concordance group}, wherein one ignores the differentiable structures and only requires the annulus to be topologically locally flatly embedded instead of smoothly embedded. Knots that bound topologically locally flat discs in a topological manifold homeomorphic to the $4$-ball  are called \emph{topologically slice knots}.  Let $\calt < \calc$ denote the subgroup of smooth concordance classes of topologically slice knots. Thus the group $\calt$ measures the difference between the smooth and topological categories. Endo established in ~\cite{Endo} that $\calt$ itself is large; it has infinite rank as an abelian group.

In~\cite{CHH:Filt}, the authors and Harvey introduced the \textit{bipolar filtration of} $\calt$:
 \[ \cdots \subset \calt_{n+1} \subset \calt_n \subset \cdots \calt_2 \subset \calt_1 \subset \calt_0 \subset \calt \] 
which is obtained as the restriction to $\calt$ of the \textit{bipolar filtration of }$\calc$, defined as follows. A knot $K$ is \emph{$n$-positive} if it bounds a smoothly embedded $2$-disk, $D$, in a smooth four-manifold $V$ (with $\partial V\cong S^3$) that satisfies:
	\begin{itemize}
		\item $\pi_1(V)=0$,
		\item the intersection form on $H_2(V)$ is positive definite, and
		\item $H_2(V)$ has a basis represented by a collection of surfaces $\{S_i\}$ disjointly embedded in the complement of the slice disc and $\pi_1(S_i)\subset \left(\pi_1(V-D)\right)^{(n)}$.
	\end{itemize}
The first two conditions are equivalent to saying that	$V$ is a smooth manifold homeomorphic to a punctured connected sum of copies of $\mathbb{C}P(2)$.  An \emph{$n$-negative} knot is defined by replacing the word `positive' with `negative' above.  The four-manifold $V$ is called an \emph{$n$-positon} (respectively, an \emph{$n$-negaton}).  The set of $n$-positive (respectively, $n$-negative) knots is denoted $\calp_n$ (respectively, $\caln_n$).  A knot is \emph{$n$-bipolar} if it is both $n$-positive and $n$-negative.  The set of concordance classes of $n$-bipolar knots forms a subgroup of $\calc$.  This induces a filtration of $\calt$ by defining $\calt_n$ as the intersection of $\calt$ with $n$-bipolar knots. It was conjectured in ~\cite{CHH:Filt} that $\calt_n/\calt_{n+1}$ is non-trivial for every $n$. As evidence, it was shown there, using a m\'elange of techniques from smooth and topological concordance,  that $\calt_1/\calt_2$ has positive rank ~\cite[Theorem 8.1]{CHH:Filt} and it was observed that  Endo's examples generate an infinite rank subgroup of $\calt/\calt_0$ ~\cite[Theorem 4.7]{CHH:Filt}.  The main result of the present paper is: 

\begin{thm}\label{thm:MainTheoremBipolar} $\calt_0/\calt_1$ has infinite rank. 
\end{thm}

\noindent Recall that any knot in $\mathcal{T}_0$ necessarily has zero $\tau$, $s$ and $\epsilon$-invariants and vanishing Fintushel-Stern-Hedden-Kirk invariant ~\cite[Prop. 1.2]{CHH:Filt}. This makes such knots difficult to detect.

This theorem is proved by considering the infinite family of knots $\{K_p\}$ below and analyzing certain $d$-invariants associated to their $2$-fold branched covers. These invariants have been shown to obstruct membership in $\calt_1$ ~\cite[Corollary 6.9]{CHH:Filt}. Our calculation relies on Ozsv\'ath-Szab\'o's $d$-invariant formula for rational surgeries on $L$-space knots. Let $D$ denote the positively-clasped, untwisted Whitehead double of the right-handed trefoil and let $U$ denote the trivial knot. Let $R_p$ denote the ribbon knot on the right side of Figure~\ref{fig:knot} (there is an unknotted, untwisted curve going once over each band on the obvious genus one Seifert surface for $R_p$).  If we tie the $(p+1)$-twisted band of $R_p$ into $D$, we obtain the topologically slice knot $K_p$, pictured on the left side of Figure~\ref{fig:knot}.

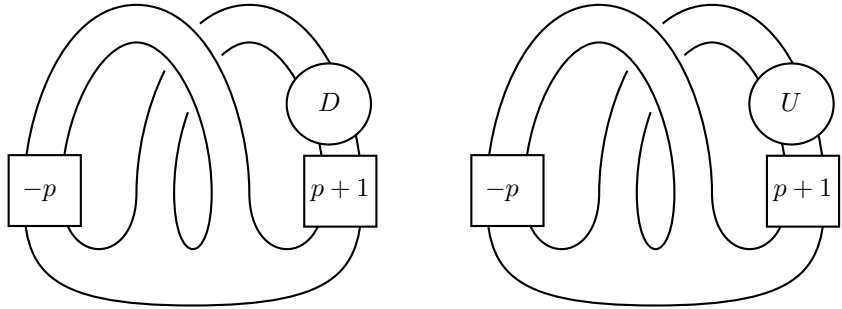
\begin{figure}[!ht]
	\begin{center}
		\begin{overpic}[scale=1]{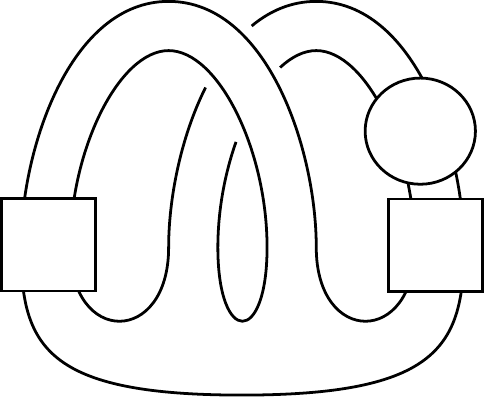}
			\put(4, 30){$-p$}
			\put(82, 30){$p+1$}
			\put(84, 53){$D$}
		\end{overpic}
		\hspace{1cm}
		\begin{overpic}{knot.pdf}
			\put(4, 30){$-p$}
			\put(82, 30){$p+1$}
			\put(84, 53){$U$}
		\end{overpic}
		\caption{The knots $K_p$ and $R_p$}
		\label{fig:knot}
	\end{center}
\end{figure}

Furthermore we strengthen Theorem~\ref{thm:MainTheoremBipolar} by showing that $\calt_0/\calt_1$ has infinite rank even ``modulo knots of Alexander polynomial one.'' Recall that any knot with Alexander polynomial one is a topologically slice knot ~\cite{FQ}. Let $\Delta$ denote the subgroup of $\calt$ consisting of knots that are smoothly concordant to Alexander polynomial one knots. Most early examples of topologically slice knots that are not smoothly slice were knots of Alexander polynomial one knots.  Indeed  recall that all of Endo's examples giving an infinite rank subgroup of $\calt/\calt_0$ have Alexander polynomial $1$ ~\cite{Endo}.    This suggested the question, ``Is $\Delta=\calt$?'' , that is to say can Alexander polynomial one knots account for all the subtlety of $\calt$?  Recently, Hedden-Livingston-Ruberman answered this question in the negative by exhibiting an infinite rank subgroup of $\calt/\Delta$ ~\cite[Theorem A]{HedLivRub}. In this direction we strengthen Theorem~\ref{thm:MainTheoremBipolar} as follows:


\begin{thm}\label{thm:MainTheoremAlex} The family $\displaystyle \left\{K_p\right\}_{p = 1}^\infty$ generates an infinite rank subgroup of $\calt_0/\langle\calt_1,\Delta\rangle$. Thus in particular
\begin{itemize}
\item the family $\displaystyle \left\{K_p\right\}_{p = 1}^\infty$ generates an infinite rank subgroup of $\calt_0/\calt_1$; and
\item the family $\displaystyle \left\{K_p\right\}_{p = 1}^\infty$ generates an infinite rank subgroup of $\calt/\Delta$.
\end{itemize}
\end{thm}

Aside from strengthening Theorem~\ref{thm:MainTheoremBipolar},  the second point of Theorem~\ref{thm:MainTheoremAlex}  gives another (shorter) proof of the above result of Hedden-Livingston-Ruberman ~\cite[Theorem A]{HedLivRub}. Moreover the $K_p$ lie in $\calt_0$ whereas the examples of ~\cite{HedLivRub} do not.

\section{$K_p$ lies in $\calt_0$}\label{sec:zerobipolar}

In this section we show that if $p\geq 3$ then the knot $K_p$ lies in $\calt_0$.

The right-handed trefoil knot can be unknotted by changing one positive crossing so it lies in $\mathcal{P}_0$ ~\cite[Lemma 3.4]{CocLic}. Hence the Whitehead double $D$ of the right-handed trefoil knot lies in $\calp_1$ by ~\cite[Example 3.5]{CHH:Filt}. Since $K_p$ is a winding number zero satellite of $D$ with pattern the ribbon knot $R_p$,  $K_p$ lies in $\calp_2$,  by ~\cite[Corollary 3.4]{CHH:Filt}, and so in particular lies in $ \calp_0$.  Furthermore, $K_p$ is topologically slice since $D$ and $R_p$ are topologically slice. Hence $K_p\in\calt$. It remains only to show that  $K_p \in \caln_0$, which requires more work.

\begin{prop}\label{knotisnegative}
	If $p \geq 3$, then $K_p \in \caln_0$.
\end{prop}
\begin{proof}
	Consider the right-hand band in the Seifert surface for the knot $K_p$ from Figure~\ref{fig:knot}, which is an annulus, $A$, whose core, $\a$, has the knot type of $D$.  The annulus $A$ is twisted in the sense that the self linking of $\a$ is ${\mathrm{sl}}(\a)=p+1$.
	
	One can change $D$ to the unknot by a single positive crossing change.  Typically one achieves this operation by blowing down a $-1$-framed $2$-handle since this does not change framings, but we will instead blow down a $+1$-framed handle.  This operation is pictured in Figure~\ref{fig:crossingchange}.  Let $BD(\a)$ and $BD(A)$ denote the results of $\a$ and $A$ after the blow down.  As a knot, $BD(\a)$ is simply the unknot, but the blow down procedure changes the framing of the annulus.  In fact, ${\mathrm{sl}}(BD(\a))= {\mathrm{sl}}(\a)-4 = p-3$.  This change in framing follows directly from the local computation pictured in Figure~\ref{fig:crossingchange}; the positive crossing turns negative (decreasing writhe by $2$), and two left-handed twists appear in the annulus.
	
	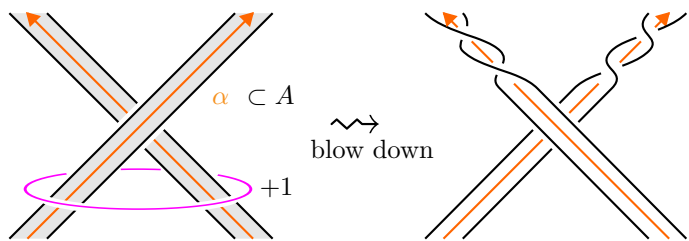
\begin{figure}[!ht]
		\begin{center}
			\begin{overpic}[scale=1]{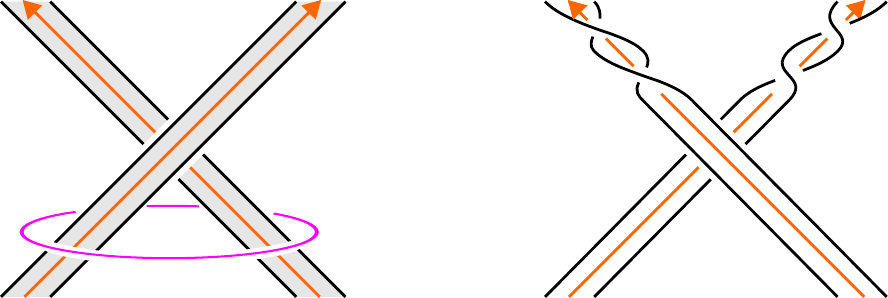}
				\put(37, 6.5){$+1$}
				\put(47, 15){\Huge $\leadsto$}
				\put(44.5, 12){blow down}
				\put(30, 20){\textcolor{BurntOrange}{$\a$} \ $\subset A$}
			\end{overpic}
			\caption{Changing a positive crossing by blowing down a $+1$}
			\label{fig:crossingchange}
		\end{center}
	\end{figure}
	
	Let $J_p$ denote the knot resulting after doing this single blow down on $K_p$.  Each of the left-hand and right-hand bands of $J_p$ is unknotted with framing $-p$ and $p-3$, respectively.  If $p\geq 3$, then $J_p$ can be unknotted by changing only negative crossings.  Hence $J_p \in \caln_0$~\cite[Lemma 3.4]{CocLic}.
	
	The cobordism $W$ from $S^3$ to $S^3$, induced by the $+1$-framed $2$-handle attachment has intersection form $\langle +1 \rangle$.  As in~\cite[Lemma 3.4]{CocLic}, we have that $K_p$ (in the $-S^3$ boundary component) is concordant to $J_p$ (in the $S^3$ boundary component) via a concordance which is disjoint from a generator of $H_2(W)$.  By reversing the orientation on $W$ and gluing on a $0$-negaton for $J_p$, we obtain a $0$-negaton for $K_p$, which finishes the proof.
\end{proof}

\section{Heegaard Floer correction terms}\label{HFCT}

In this section we introduce the $d$-invariants we will use to prove our main theorems and calculate them on our examples. These calculations are then called in the proofs in the next section.

Given a rational homology sphere $Y$ and a $\spinc$ structure $\sss$ on $Y$, Ozsv\'ath and Szab\'o defined the \emph{correction term (or $d$-invariant)} $d(Y,\sss) \in \bbq$~\cite[Definition 4.1.]{OzSz:AbsGrad}.  If $K$ is a knot in $S^3$ and $Y$ is a prime-power cyclic cover of $S^3$ branched along $K$, then $Y$ is a rational homology sphere~\cite{CassGord:SliceKnots}.  The correction terms obstruct a knot's lying in the first term of the bipolar filtration:

\begin{thm}\cite[Theorem 6.5]{CHH:Filt}\label{thm:T1Obstruction}
	If $K \in \caln_1$ and $Y$ is any prime-power cover of $S^3$ branched along $K$, then there exists a metabolizer\footnote{A metabolizer is a square-root order subgroup of $H_1(Y)$ on which the form $\mathrm{lk}:H_1(Y)\times H_1(Y) \to \bbq/\bbz$ vanishes.} $G < H_1(Y)$ for the $\bbq/\bbz$-valued linking form on $H_1(Y)$, and there is a $\spinc$ structure $\sss_0$ on $Y$ such that for each $z\in G$
	\[ d(Y, \sss_0 + \widehat{z}) \geq 0 \]
	where $\widehat{z}$ denotes the Poincar\'e dual of $z$.  For instance, we may take $\sss_0$ to be the $\spinc$ structure corresponding to any $\mathrm{spin}$ structure on $Y$.
\end{thm}

There is an analogous result if $K\in \calp_1$ ~\cite[Theorem 6.2]{CHH:Filt}. This motivates the following calculation of the correction terms for $\Sigma_p$, the $2$-fold cover of $S^3$ branched along $K_p$.

According to the Akbulut-Kirby method~\cite{AkKir}, the $2$-fold branched cover over our knot $K_p$ has a surgery diagram as in Figure~\ref{fig:DoubleCover} (we have used the fact that $D$ is reversible).  One may eliminate the unknotted component by the `slam-dunk' move ~\cite[Figure 5.30]{GompfStip} to see that $\Sigma_p = S^3_{r}(D\# D)$ where $r = (2p+1)^2/2p$.

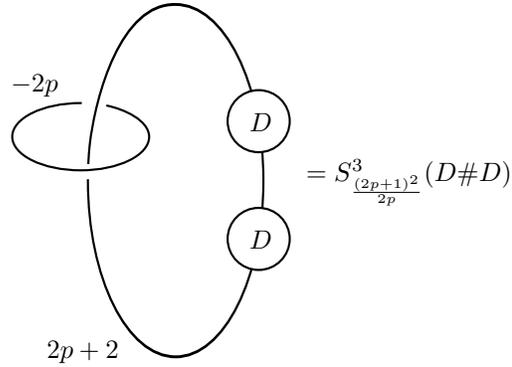
\begin{figure}[!ht]
	\begin{center}
		\begin{overpic}{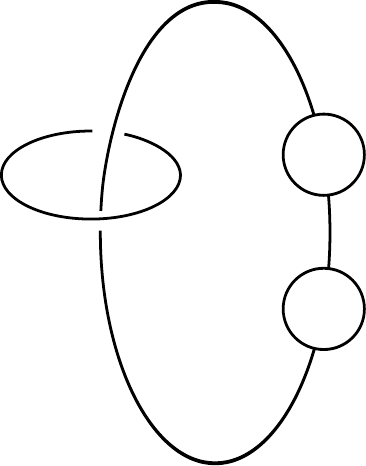}
			\put(0, 75){$-2p$}
			\put(67, 64){$D$}
			\put(67, 31){$D$}
			\put(10, 0){$2p+2$}
			\put(83, 50){$= S^3_{\frac{(2p+1)^2}{2p}}(D\# D)$}
		\end{overpic}
		\caption{The double branched cover $\Sigma_p$ of $K_p$}
		\label{fig:DoubleCover}
	\end{center}
\end{figure}
By Appendix~\ref{sec:cfkinfty}, the correction terms of any positive rational surgery on $D\# D$ agree with those of the same surgery on the torus knot $T_{2,5}$. Moreover every positive torus knot admits lens space surgery~\cite{Moser}, so we may apply~\cite[Theorem 1.2]{OzSz:RatSurg} to calculate the correction terms of $p/q$ surgery on $T_{2,5}$ as long as $p/q \geq 2\,g(T_{2,5})-1 = 3$:
\begin{equation}\label{eq:CTOLSK}
	d\left( S^3_{p/q} \left( T_{2,5} \right) , i \right) - d\left( S^3_{p/q}(U) , i \right) = -2\,t_{\left| \left\lfloor i/q \right\rfloor \right|} \left( T_{2,5} \right)
\end{equation}
for each $|i| \leq p/2$, and where $t_i \left( T_{2,5} \right)$ is the \emph{torsion coefficient}, which is determined by the symmetrized Alexander polynomial $\displaystyle \Delta\left( T_{2,5} \right) = a_0 + \sum_{i > 0} a_i\left( t^i + t^{-i} \right) = 1 - \left(t + t\inv \right) + \left(t^2 + t^{-2} \right)$ according to \[ t_i\left( T_{2,5} \right) = \sum_{j > 0} j\,a_{|i|+j} \]  The torsion coefficients of $T_{2,5}$ are $t_0 = 1$, $t_1 = 1$, and $t_i = 0$ for $i \geq 2$.  We remark that the `$i$' in Equation~\ref{eq:CTOLSK} is not a $\spinc$ structure but rather is an integer that labels a certain $\spinc$ structure.  This subtle point is discussed further in Appendix~\ref{app:spinc}.

We return to the calculation of the correction terms for $\Sigma_p$.  We require the technical assumption that $2p+1$ is prime, so that $H_1 \left( \Sigma_p \right) \cong \bbz/(2p+1)^2\bbz$ has a unique subgroup of order $2p+1$ (this subgroup is a metabolizer for the linking form).  Let $G < H_1(Y)$ denote this subgroup.  Following the notation of Theorem~\ref{thm:T1Obstruction}, we will calculate the correction terms $d(\Sigma_p, \sss_0 + \widehat{z})$ for each $z\in G$.  By Lemma~\ref{lem:CentralLabel} and the discussion in Appendix~\ref{app:spinc}, the $\spinc$ structures $\sss_0 + \widehat{z}$ have labels $(2p^2 + 2p + 1)(2p - 1) + (2p+1)k$ for $k=0,1,\ldots{2p}$.  In order to to use Equation~\ref{eq:CTOLSK}, we need each label $i$ to satisfy $|i| \leq (2p+1)^2/2$.  These labels range from $4p^3+2p^2-1$ to $4p^3+6p^2+2p-1$ and must therefore be shifted down by a suitable multiple of $(2p+1)^2$.  Subtract $p(2p+1)^2$ from each label to arrive at the new labels 
\[ i_k := -2p^2 -p -1 + k (2p + 1) \hspace{1cm} k = 0,1,\ldots 2p\]
One easily checks that each of these new labels is less than $(2p+1)^2/2$ in absolute value.  Note that $i_0$ is the label of the $\spinc$ structure $\sss_0$ which corresponds to the unique $\spin$ structure on $\Sigma_p$.

\begin{lem}\label{lem:CorTermCalc}
	Let $i_k$ denote the label $-2p^2 -p -1 + k (2p + 1)$ for $k=0,1,\ldots,2p$.  The correction terms $d \left( \Sigma_p, i_k \right)$ are
	\[ d \left( \Sigma_p, i_k \right) = \begin{cases}
											-2, & \textrm{if }k=p,p+1\\
											0, & \textrm{otherwise}
										\end{cases}\]
\end{lem}

Note that it follows immediately from Theorem~\ref{thm:T1Obstruction} that $K_p$ is not in $\caln_1$, if $2p+1$ is prime, hence not in $\calt_1$.

\begin{proof}
	The correction terms $d\left( S^3_{(2p+1)^2/2p}(U) , i_k \right)$ must vanish, for $S^3_{(2p+1)^2/2p}(U)$ is the double branched cover of the ribbon knot $R_p$, and the correction terms of the double branched cover of a ribbon knot (with $\spinc$ structures corresponding to a metabolizing subgroup of $H_1$) must vanish~\cite[Theorem 2.3]{JabNai}.  Thus, by Equation~\ref{eq:CTOLSK}, $d\left( S^3_{(2p+1)^2/2p}(T_{2,5}) , i_k \right) = -2\,t_{\left| \left\lfloor i_k/2p \right\rfloor \right|} \left( T_{2,5} \right)$.
	
	Now 
	\begin{eqnarray*}
		\left\lfloor \frac{i_k}{2p} \right\rfloor &=& \left\lfloor \frac{-2p^2-p-1 + (2p+1)k}{2p} \right\rfloor\\
		&=& \left\lfloor -p +k + \frac{k-p-1}{2p} \right\rfloor\\
		&=& -p+k + \left\lfloor \frac{k-p-1}{2p} \right\rfloor\\
		&=& \begin{cases}
				k-p-1, & \textrm{if }0\leq k \leq p\\
				k-p, & \textrm{if }p<k\leq 2p
			\end{cases}
	\end{eqnarray*}
	
	Recall that $t_j(T_{2,5})$ is nonzero if and only if $j=0$ or $j=1$, so it suffices to find the $j=\left|\left\lfloor i_k/2p \right\rfloor \right|$ equal to $0$ or $1$.  By the above calculation, this happens only if $k=p$ or $p+1$, in which case $\left| \left\lfloor i_p/2p \right\rfloor \right| = \left| \left\lfloor i_{p+1}/2p \right\rfloor \right| = 1$.
	
	The lemma follows.  (Note that the $\spinc$ structures $i_p$ and $i_{p+1}$ are `conjugate' $\spinc$ structures; this is reflected by the fact that $- p(2p+1) \equiv (p+1)(2p+1) \textrm{ mod } (2p+1)^2$.)
\end{proof}


%

\section{Proof of Theorems~\ref{thm:MainTheoremBipolar} and~\ref{thm:MainTheoremAlex} }\label{sec:proof}

Let $\calp$ be the subset of natural numbers $p \geq 3$ such that $2p+1$ is prime.  We showed in Section~\ref{sec:zerobipolar} that each $K_p\in\calt_0$.  We will show that the set $\{K_p\}_{p\in\calp}$ is linearly independent in $\calt_0/\calt_1$. Consider a linear combination: \[ K := \consum_{p\in\calp} n_p \, K_p \]   Suppose $K\in \calt_1$.  Without loss of generality, we may assume that some $n_q > 0$.  We saw in Section~\ref{sec:zerobipolar} that $K_q \in \calp_2 \subset \calp_1$, so $-K_q \in \caln_1$.  Thus $J := K \# -(n_q - 1)K_q \in \caln_1$, which we will endeavor to contradict.  Note that \[ J = K_q \consum_{p\in\calp'} n_p\,K_p \] where $\calp' = \calp  - \{q\}$.

If $\Sigma$ denotes the double branched cover of $J$, then \[ \Sigma = \Sigma_q\consum_{p\in\calp'} n_p \, \Sigma_p \] as an oriented manifold.  It is known that a metabolizer $M$ for $H_1(\Sigma)$ splits as $M = M_q\oplus_{p\in\calp'} M_p$ where $M_p$ is a metabolizer for $n_p \, \Sigma_p$, since the primes $2p+1$ are distinct.  Recall that $\spinc$ structures on a connected sum split into `sums' of $\spinc$ structures on each summand, and the $d$-invariants respect this decomposition~\cite[Theorem 4.3]{OzSz:AbsGrad}. Since $H_1(\Sigma_q)$ has order the square of the prime $2q+1$, $M_q$ is necessarily the subgroup generated by $2q+1$. In particular, for each $k=0,\ldots,2q$ the element $(k(2q+1), \vec 0)\in H_1(\Sigma)\cong \bbz/(2q+1)^2\bbz \,\bigoplus_{p \in \calp'}\, \left( \bbz/(2p+1)^2\bbz \right)^{|n_p|}$ lies in the metabolizer $M$.  We are interested in the $\spinc$ structure $\sss = \sss_0 + \mathrm{PD}((k_q(2q+1), \vec 0)) = \left(\sss_0+\mathrm{PD}(k_q(2q+1))\right) \# \sss_0$ on $\Sigma_q \consum_{p\in\calp'} n_p \, \Sigma_p$ for which the label of $\sss_0+ \mathrm{PD}(k_q(2q+1))\in \spinc(\Sigma_q)$ is the integer $i_q$ as in Lemma~\ref{lem:CorTermCalc}.   We can now calculate: 

\begin{eqnarray*}
	d(\Sigma, \sss) &=& d\left(\Sigma_q \consum_{p\in\calp'} n_p \, \Sigma_p , (\sss_0+\mathrm{PD}(k_q(2q+1)))\ \#\ \sss_0\right)\\
	&=& d\left(\Sigma_q, \sss_0 + \mathrm{PD}(k_q(2q+1))\right) + d\left(\consum_{p\in\calp'} n_p \, \Sigma_p , \sss_0\right)\\
	&=& d\left( \Sigma_q, i_q \right) + \sum_{p\in\calp'} n_p \, d\left( \Sigma_p, i_0 \right)\\
	&=& -2
\end{eqnarray*}
where the last two lines follow (notationally and logically) from Lemma~\ref{lem:CorTermCalc}.

By Theorem~\ref{thm:T1Obstruction}, we have contradicted that $J \in \caln_1$, and we conclude that the $K_p$ are linearly independent in $\calt_0/\calt_1$, finishing the proof of Theorem~\ref{thm:MainTheoremBipolar}.

More generally, to see that $\{K_p\}$ is linearly independent in $\calt/\langle\calt_1,\Delta\rangle$, suppose that $\displaystyle T\, \consum_{p\in\calp} n_p \, K_p\in \calt_1$  for some $T$ with $\Delta_T(t)=1$.  Again we may assume $n_q > 0$ for some $q$. Since  $\pm K_p\in\calt_0$ it follows that $T\in\calt_0$. Since $\Delta_T(t)=1$, $\Sigma_T$ is a homology sphere. Hence, by ~\cite[Corollary 6.11]{CHH:Filt}, $d(\Sigma_T, \sss_0) = 0$.  Now, as above, we may add enough mirror images of $K_q$ to arrive at the conclusion \[ J := T\, \#\, K_q\, \consum_{p\in\calp'} n_p\,K_p \in \caln_1 \]  Using the same $d$-invariant argument as for the $\calt_0/\calt_1$ case, together with the facts that $H_1(\Sigma_T)=0$ and $d(\Sigma_T, \sss_0) = 0$, we can, given any metabolizer for $H_1(\Sigma_J)$, find a $\spinc$ structure $\sss$ for $\Sigma_J$ which corresponds to some element in that metabolizer and which satisfies $d(\Sigma_J, \sss) = -2$.  By Theorem~\ref{thm:T1Obstruction}, this contradicts that $J \in \caln_1$.  This finishes the proof of Theorem~\ref{thm:MainTheoremAlex}.


\section{Appendix on $CFK^\infty(D\# D)$}\label{sec:cfkinfty}

Recall that $D$ denotes the untwisted, positively clasped Whitehead double of $T_{2,3}$.  In this appendix, we establish

\begin{lem}\label{lem:2DisT25}
	For any $p/q \in \bbq - \{0\}$ and $0\leq i \leq p-1$, $d\left( S^3_{p/q}(D\# D), i\right) = d\left(S^3_{p/q}(T_{2,5}), i\right)$.
\end{lem}

To any knot $K$ in $S^3$, one can associate the $\bbz\oplus\bbz$-filtered knot Floer chain complex $\cfkinfty(K)$ introduced by Rasmussen~\cite{Ras:Thesis} and Ozsv\'ath and Szab\'o~\cite{OzSz:KnotInv}.  It is known that the complex $\cfkinfty(D)$ splits as \[ \cfkinfty(D) = \cfkinfty(T_{2,3}) \oplus A \] where $A$ is an acyclic complex~\cite[Appendix 9.1]{CHH:Filt} (cf.~\cite[Lemma 6.12]{Hom:CFKandConc}).  The complex of the connected sum of knots is the tensor product of the individual knots' complexes~\cite[Theorem 7.1]{OzSz:KnotInv}, so we see that \[ \cfkinfty(D\# D) = \left(\cfkinfty(T_{2,3})\otimes\cfkinfty(T_{2,3})\right)\oplus A \] where $A$ is acyclic.  The tensor product of the $T_{2,3}$ complexes is (up to an acyclic summand) $\cfkinfty(T_{2,5})$.  This is most easily seen by a direct computation.  The complexes are pictured in Figure~\ref{fig:complexes} (for simplicity, we have collapsed all $U$ and $U\inv$ translates in the figure).

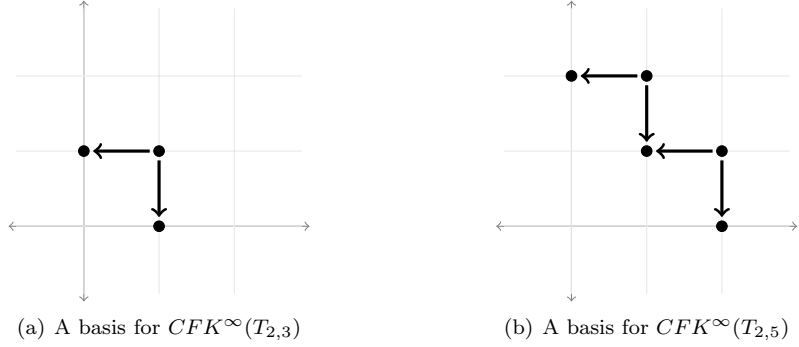
\begin{figure}[!ht]
	\centering
			\subfigure[A basis for $CFK^\infty(T_{2,3})$]
			{
				\begin{tikzpicture}
					\begin{scope}[thin, gray]
						\draw [<->] (-1, 0) -- (3, 0);
						\draw [<->] (0, -1) -- (0, 3);
					\end{scope}
					\draw[step=1, black!10!white, very thin] (-0.9, -0.9) grid (2.9, 2.9);
					\filldraw (0, 1) circle (2pt) node[] (x0){};
					\filldraw (1, 1) circle (2pt) node[] (x1){};
					\filldraw (1, 0) circle (2pt) node[] (x2){};
					\draw [very thick, <-] (x0) -- (x1);
					\draw [very thick, <-] (x2) -- (x1);
				\end{tikzpicture}
			}
			\hspace{2cm}
			\subfigure[A basis for $CFK^\infty(T_{2,5})$]
			{
				\begin{tikzpicture}
					\begin{scope}[thin, gray]
						\draw [<->] (-1, 0) -- (3, 0);
						\draw [<->] (0, -1) -- (0, 3);
					\end{scope}
					\draw[step=1, black!10!white, very thin] (-0.9, -0.9) grid (2.9, 2.9);
					\filldraw (0, 2) circle (2pt) node[] (x0){};
					\filldraw (1, 2) circle (2pt) node[] (x1){};
					\filldraw (1, 1) circle (2pt) node[] (x2){};
					\filldraw (2, 1) circle (2pt) node[] (x3){};
					\filldraw (2, 0) circle (2pt) node[] (x4){};
					\draw [very thick, <-] (x0) -- (x1);
					\draw [very thick, <-] (x2) -- (x1);
					\draw [very thick, <-] (x2) -- (x3);
					\draw [very thick, <-] (x4) -- (x3);
				\end{tikzpicture}
			}
			\caption{The knot Floer complexes of $T_{2,3}$ and $T_{2,5}$}
			\label{fig:complexes}
\end{figure}



Ozsv\'ath and Szab\'o proved~\cite[Theorem 4.4]{OzSz:KnotInv} that given a sufficiently large, integral surgery coefficient $r$, $HF^+(K_r, \sss)$ is isomorphic to the homology of a quotient complex of $\cfkinfty(K)$.  (In fact, the same is true for a sufficiently negative, integral coefficient -- one must use a different quotient complex.  The argument below works for either case.)  Recall that \[ d(K_r, \sss) := \min_{0 \neq \a \in HF^+(K_r, \sss)} \left\{\mathrm{gr}(\a):\ \a\in\mathrm{im}\ U^k \textrm{ for all }k\geq 0\right\} \]

If $A$ is an acyclic summand of $\cfkinfty(K)$ (as a $\bbz/2\bbz[U,U\inv]$-module), we claim that $A$ does not affect the $d$-invariants.  Let $Q$ denote the quotient complex of $\cfkinfty(K)$ whose homology is isomorphic to $HF^+(K_r,\sss)$.  By definition~\cite[Theorem 4.4]{OzSz:KnotInv}, $Q$ is the quotient of $\cfkinfty(K)$ by a subcomplex that is $U$-invariant.  Suppose $a\in A$ with $0\neq [a] \in H_\ast(Q)$.  Let $\del$ denote the differential in $\cfkinfty(K)$ and $\del_Q$ the induced differential in $Q$.  If $[a] \in \mathrm{im}\ U^k$ for all $k\geq 0$, then there is a $B\in\cfkinfty(K)$ and a $k \geq 0$ with $U^k[B]=[a]\in H_\ast(Q)$ and $\del B = \del_Q B = 0$.  On the chain level, we have $U^k\,B = a + \del_Q x$ for some $x\in Q$, and so $U^k\,B = a + \del x + y \in \cfkinfty(K)$ for some $y\in\ker(\mathrm{projection})$.  Thus $0 = \del B = U^{-k} \left( \del a + \del^2 x + \del y \right) = U^{-k}\left( \del a + \del y \right)$, from which it follows that $\del a = \del y$.  Since $a$ lies in the summand $A$, so does $y$.  Then $a+y$ is a cycle in $\cfkinfty(K)$ and must be a boundary since $A$ is acyclic.  If $\del \theta = a+y$, then $\del_Q \theta = a$, and so $0=[a]=\in H_\ast(Q)$, a contradiction.

Thus, acyclic summands in $\cfkinfty$ do not contribute to $d$-invariants of sufficiently large  or sufficiently negative integer surgeries.  In fact, acyclic summands do not affect the $d$-invariants of any nonzero rational surgery (see the discussion surrounding~\cite[Proposition 2.2]{RubStr} and in particular the proof of~\cite[Corollary 2.3(1)]{RubStr}).

\section{Appendix on $\spinc$ structures of lens spaces}\label{app:spinc}

Recall that $S^3_{p/q}(U)=-L(p,q)$.  Ozsv\'ath and Szab\'o give a recursive formula for the $d$ invariants of lens spaces~\cite[Equation 12]{OzSz:AbsGrad}.  Preceding their proof, they give an ordering (or labeling) of the $\spinc$ structures of $-L(p,q)$ in terms of its standard genus one Heegaard diagram $(\mathbb{T}^2, \a, \gamma, z)$ (in the notation of~\cite[Figure 2]{OzSz:AbsGrad} or Figure~\ref{fig:L53}).  Recall that for such a pointed Heegaard diagram, a point in $\a\cap\gamma$ determines a $\spinc$ structure on $-L(p,q)$.  Ozsv\'ath and Szab\'o label the $p$ intersection points of $\a$ and $\gamma$ circularly about $\a$ (using the basepoint $z$ to determine which point is labeled by $0$), and hence give a labeling of the $\spinc$ structures $\ell: \bbz/p\bbz = \{0, 1, \ldots p-1\} \to \spinc(-L(p,q))$.   This construction extends (by using a Heegaard triple diagram  adapted to the surgery on a knot) to give an identification of $\spinc(S^3_{p/q}(K)) \leftrightarrow \bbz/p\bbz$ for any $K$.  Let $\mu$ denote the closed curve comprising the left edge of the picture of the torus in Figure~\ref{fig:L53}; an orientation on $\mu$ determines a generator of $H_1(-L(p,q))\cong \bbz/p\bbz$.

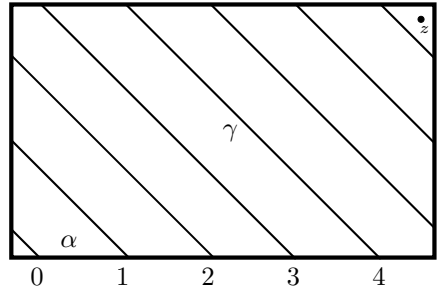
\begin{figure}[!ht]
	\begin{center}
		\begin{overpic}{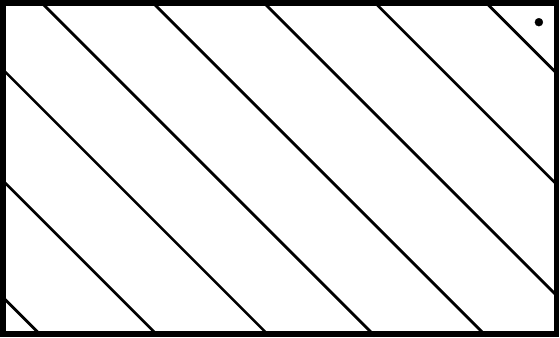}
			\put(5, -6){$0$}
			\put(25, -6){$1$}
			\put(45, -6){$2$}
			\put(65, -6){$3$}
			\put(85, -6){$4$}
			\put(96,53){\tiny $z$}
			\put(50,29){$\gamma$}
			\put(12,3){$\a$}
		\end{overpic}
		\caption{The labeling of intersection points on the Heegaard diagram for $-L(5,3)$}
		\label{fig:L53}
	\end{center}
\end{figure}

The set $\spinc(-L(p,q))$ is an affine set over $H^2(-L(p,q))$ (cf.~\cite[Section 2.6]{OzSz:3Man}), which allows one to take the `difference' of $\spinc$ structures.  If $\sss(\mathbf{x})$ and $\sss(\mathbf{y})$ are the $\spinc$ structures corresponding to $\mathbf{x}$ and $\mathbf{y}\in \a\cap\gamma$, then
\begin{equation}\label{eq:SpinCDifference}
	\sss(\mathbf{y}) - \sss(\mathbf{x}) = \mathrm{PD}(\epsilon(a - c))
\end{equation}
where $a$ (respectively, $c$) is a 1-chain corresponding to a path from $\mathbf{x}$ to $\mathbf{y}$ in $\a$ (respectively, $\gamma$) and $\epsilon$ is the map $H_1(\mathbb{T}^2) \twoheadrightarrow H_1(L(p,q))$ induced by attaching the $\a$ and $\gamma$ handles~\cite[Lemma 2.19]{OzSz:3Man}.

Since $q$ is relatively prime to $p$, there exists $k = q\inv \pmod p$.  One may check easily that the homology class $\epsilon(i,i+1) = k\cdot\mu$ where $\mu$ is the generator of $H_1(-L(p,q))$ discussed above (after picking the appropriate orientation).  See Figure~\ref{fig:L53} for an example.  One may generalize this fact to see $\epsilon(i,j)=(j-i)k\cdot\mu$ for any $i,j\in\bbz/p\bbz$.  By Equation~\ref{eq:SpinCDifference} \[ \ell(j) - \ell(i) = (j-i)k\cdot\mathrm{PD}(\mu) \]  In this sense, the circular ordering of (the labels of the) $\spinc$ structures affinely respects the group structure of $H^2(-L(p,q))$. In other words, if we fix the affine identification $\phi_0:\spinc(-L(p,q)) \to H^2(-L(p,q))$ with $\phi_0(\ell(0)) = 0$, then then the composition  $\phi_0\circ\ell: \bbz/p\bbz \to H^2(-L(p,q))$ is the group isomorphism given by $1\mapsto k\cdot\mathrm{PD}(\mu)$.

In $p$ is odd, there is a unique $\spin$ structure on $-L(p,q)$, which in turn determines a $\spinc$ structure $\sss'$, called the \emph{central $\spinc$ structure}.  The central $\spinc$ structure is the only $\spinc$ structure with trivial $c_1$.  In this case $c_1$ is a one-to-one correspondence.  \[\bbz/p\bbz \xrightarrow{\ \ \ell\ \ } \spinc(-L(p,q)) \xrightarrow{\ \ c_1\ \ } H^2(-L(p,q)) \]

In order to use Equation~\ref{eq:CTOLSK} for those $\spinc$ structures corresponding to a metabolizer for $H_1(-L(p,q))$, we must determine the Ozsv\'ath-Szab\'o labels of these $\spinc$ structures.  Since $p = t^2$ for some odd prime $t$, there is only one subgroup of $H^2$ which has order $t$.  Thus, the labels of the $\spinc$ structures corresponding to this subgroup are $i'+ m\,t$ ($m=0,\ldots,t-1$) where $i' = \ell\inv(\sss')$ is the label of the central $\spinc$ structure $\sss'$. 

It remains to determine $i'$.

\begin{lem}\label{lem:CentralLabel}
	For $p$ odd, the label $i'$ of the central $\spinc$ structure $\sss'$ is given by the mod $p$ reduction of $\displaystyle \frac{(p+1)(q-1)}{2}$.
\end{lem}
\begin{proof}
	Throughout the proof we will ignore orientations of the boundary $3$-manifolds since we are only concerned with the zero element of $H^2$.
	
	Let $p > q > t > 0$ where $t$ is the mod $q$ reduction of $p$.  Then $p = nq + t$ for some $n > 0$.  $L(q,t)$ has a Dehn surgery description as $q/t$ surgery on the unknot.  Add a $4$-dimensional $2$-handle along the meridian with framing $-n$.  This determines a cobordism $W$ to $L(p,q)$ (to see this, perform the slam-dunk move~\cite{GompfStip}). Note that $H_1(W)=0$ and $H_2(W)\cong \mathbb{Z}$. It follows from duality and the universal coefficient theorem that that both $H^2(W)$ and $H^2(W,\partial(W))$ are infinite cyclic; and so the exact sequence on relative homology is \[ 0 \to H^2(W, \del W) \xrightarrow{\cdot pq} H^2(W) \twoheadrightarrow H_1(L(p,q))\oplus H_1(L(q,t)) \to 0 \]
	
	In the proof of~\cite[Proposition 4.8]{OzSz:AbsGrad}, Ozsv\'ath and Szab\'o discuss a $\spinc$ structure $\sss_z(\psi_i)$ on $W$ that restricts to $\sss_i$ on $L(p,q)$ that satisfies \[ \langle c_1(\sss_z(\psi_i)), H(\calp) \rangle = 2i + 1 - p - q \] where $H(\calp)$ is the generator of $H_2(W)$.
	
	Let $\xi$ denote a generator of $H^2(W)$. Since the Kronecker evaluation map is, in this case, an isomorphism, we can choose, $\xi$, a generator of $H^2(W)$ such that  $\langle \xi, H(\calp) \rangle = 1$. It follows that  $c_1(\sss_z(\psi_i)) = (2i+1-p-q)\xi$.
	
To see which $\sss_z(\psi_i)$ restricts to the central $\spinc$ structure on $L(p,q)$, we need to find the $i$ such that $j_1(c_1(\sss_z(\psi_i))) = 0$ in $H^2(L(p,q))$, where $j_1$ is the restriction to $L(p,q)$.  Note that if we take $i$ to be the mod $p$ reduction of $(p+1)(q-1)/2$ then $\langle c_1(\sss_z(\psi_i)), H(\calp) \rangle$ is a multiple of $p$ and so $c_1(\sss_z(\psi_i))$ is necessarily a multiple of $p\,\xi$. It follows that $j_1(c_1(\sss_z(\psi_i))) = 0$ in $H^2(L(p,q))$. Therefore the restriction of $\sss_z(\psi_i)$ to $-L(p,q)$ has $c_1=0$, which identifies this $\spinc$ structure as $s'$.
\end{proof}

\bibliographystyle{amsalpha}
\bibliography{PeterHornBib}
\end{document}